\documentclass{elsarticle}

\usepackage{hyperref}

\usepackage{amssymb,amsmath,amsfonts}

\newtheorem{theorem}{Theorem}[section]

\newtheorem{lemma}[theorem]{Lemma}

\newtheorem{definition}[theorem]{Definition}

\newtheorem{example}{Example}[section]
\newtheorem{remark}{Remark}[section]
\newtheorem{proposition}[theorem]{Proposition}

\def\normtwo#1{\Vert#1\Vert_2}

\newcommand{\eps}{\varepsilon_M}

\begin{document}

\begin{frontmatter}

\title{On computing  the symplectic  $LL^T$  factorization}

\author[MB]{Maksymilian Bujok\corref{cor1}}
\ead{mbujok@swps.edu.pl}

\author[AS]{Alicja Smoktunowicz \corref{cor2}}
\ead{A.Smoktunowicz@mini.pw.edu.pl}

\author[MB]{Grzegorz Borowik}
\ead{gborowik@swps.edu.pl}

\cortext[cor2]{Principal corresponding author}
\cortext[cor1]{Corresponding author}

\address[MB]{SWPS University of Social Sciences and Humanities,  Chodakowska 19/31, 03-815 Warsaw, Poland}
\address[AS]{ Faculty of Mathematics and Information Science, Warsaw University of Technology, Koszykowa 75, 00-662 Warsaw, Poland}

\begin{abstract} 
We analyze  two  algorithms  for computing   the symplectic $LL^T$ factorization $A=LL^T$ of a given symmetric positive definite symplectic matrix $A$.
The first algorithm $W_1$ is an implementation of the $HH^T$ factorization from \cite{Dopico}, see Theorem 5.2. The second  one, algorithm $W_2$ 
uses both Cholesky and Reverse Cholesky decompositions of symmetric positive definite matrices.

We presents a comparison of these  algorithms and illustrate  their  properties  by numerical experiments in \textsl{MATLAB}. 
A particular emphasis is given on simplecticity properties of the computed matrices  in floating-point arithmetic.
\end{abstract}


\end{frontmatter}

\section{Introduction} 

We study  numerical properties  of two algorithms for computing symplectic  $LL^T$  factorization of a given  symmetric positive definite symplectic matrix $A \in  \mathbb R^{2n \times 2n}$.
A symplectic  factorization  is the  factorization $A=LL^T$, where $L  \in \mathbb R^{2n \times 2n}$ is  block lower triangular and is symplectic. 

Let 
\begin{equation}\label{J}
J_n=\left(
\begin{array}{cc}
 0  &   I_n  \\
 -I_n &   0  
\end{array}
\right),
\end{equation}
where $I_n$ denotes the $n \times n$ identity matrix.  

We will write $J$ and $I$ instead of $J_n$ and $I_n$ when the sizes are clear from the context.

\begin{definition}\label{def1}
A matrix $A \in \mathbb R^{2n \times 2n}$ is symplectic if and only if   $A^TJA=J$.
\end{definition}

\medskip

We can use the symplectic  $LL^T$  factorization to compute the symplectic $QR$ factorization and the  Iwasawa decomposition of a given 
symplectic matrix via Cholesky decomposition. We can modify Tam's method, see \cite{Benzi}, \cite{Tam}.
Symplectic matrices arise in several applications, among which symplectic formulation of classical mechanics and quantum mechanic, 
quantum optics,  various aspects of mathematical physics, including the application of symplectic block matrices to special relativity,
optimal control theory. For more details we refer the reader to \cite{Dopico}, \cite{Benzi}, and  \cite{ Volker}.

Partition  $A \in \mathbb R^{2n \times 2n}$  conformally with $J_n$  defined by (\ref{J}) as    
\begin{equation}\label{blockA}
A=\left(
\begin{array}{cc}
 A_{11}  &  A_{12} \\
 A_{21} &   A_{22} 
\end{array}
\right),
\end{equation}
in which $A_{ij} \in \mathbb R^{n \times n}$ for $i, j=1,2$.

An immediate consequence of Definition \ref{def1} is  that the matrix $A$, partitioned as in (\ref{blockA}),  is symplectic if and only if $A_{11}^T A_{21}$ and $A_{12}^T A_{22}$ are symmetric and $A_{11}^T A_{22}-A_{21}^T A_{12}=I$. 

Symplectic matrices form a Lie group under matrix multiplications. The product $A_1  A_2$ of two symplectic matrices  $A_1, A_2 \in \mathbb R^{2n \times 2n}$ is also a symplectic matrix.
The symplectic group is closed under transposition. If $A$ is symplectic then the inverse of $A$ equals $A^{-1}$ $=J^TA^TJ$, and $A^{-1}$ is also symplectic.

Lemmas \ref{lemma1}--\ref{lemma5} will  be helpful  in the construction and for testing of some herein proposed algorithms.

\begin{lemma}\label{lemma1}
A  nonsingular block lower triangular  matrix $L\in \mathbb R^{2n \times 2n}$, partitioned as 
\begin{equation}\label{L}
L=\left(
\begin{array}{cc}
 L_{11} &   0 \\
 L_{21} &  L_{22}
\end{array}
\right),
\end{equation}
is symplectic  if and only if  $L_{22}=L_{11}^{-T}$ and $L_{21}^T L_{11}=L_{11}^T L_{21}$.
\end{lemma}

\medskip

\begin{lemma}\label{lemma2}
A matrix $Q \in \mathbb R^{2n \times 2n}$  is orthogonal  symplectic (i.e., $Q$ is both symplectic and orthogonal)  if and only if $Q$ has a form 
\begin{equation}\label{CS}
Q=\left(
\begin{array}{cc}
 C  &  S \\
 -S &   C
\end{array}
\right),
\end{equation}
where $C, S \in \mathbb R^{n \times n}$ and $U=C+i S$ is unitary. 
\end{lemma}

\medskip

\begin{lemma}\label{lemma3}
Every  symmetric positive definite symplectic matrix  $A \in \mathbb R^{2n \times 2n}$ has a spectral decomposition
$A=Q \text{diag}(D,D^{-1}) Q^T$, where $Q \in \mathbb R^{2n \times 2n}$ is orthogonal  symplectic,  and 
$D=\text{diag}(d_i)$, with $d_1 \ge d_2  \ge \ldots \ge d_n \ge 1$.
\end{lemma}

\medskip

In order to create  examples of symmetric positive definite symplectic matrices we can use the following result  from \cite{Dopico}, Theorem 5.2.

\begin{lemma}\label{lemma4}
Every  symmetric positive definite symplectic matrix  $A \in \mathbb R^{2n \times 2n}$
can be written as
\begin{equation}\label{Charlie}
A=\left(
\begin{array}{cc}
 I_n  &  0 \\
 C &   I_n
\end{array}
\right)
\quad
\left(
\begin{array}{cc}
 G  &  0 \\
 0 &   G^{-1}
\end{array}
\right)
\quad
\left(
\begin{array}{cc}
 I_n  &  C\\
 0 &   I_n
\end{array}
\right),
\end{equation}
where $G$ is symmetric positive definite and $C$ is symmetric.
\end{lemma}

\medskip

\begin{lemma}\label{lemma5}
Let  $A \in \mathbb R^{2n \times 2n}$  be a symmetric positive definite symplectic matrix, partitioned as in (\ref{blockA}). 
Let $S$ be the Schur complement of $A_{11}$ in $A$:
\begin{equation}\label{S1}
S=A_{22}-A_{12}^TA_{11}^{-1}A_{12}.
\end{equation}

Then $S$ is symmetric positive definite and we have
\begin{equation}\label{S2}
S=A_{11}^{-1}.
\end{equation}
\end{lemma}

\begin{proof}
The property (\ref{S2}) was proved in a more general setting in \cite{Dopico}, see Corollary 2.3.
We propose an alternative proof for completeness.

It is well known that if $A$ is a symmetric positive definite matrix then the Schur complement $S$ is also symmetric positive definite.
We only need to prove (\ref{S2}).  Let $A$ be  a symmetric positive definite matrix. Then $A$ is  symplectic if and only if $AJA=J$, 
which is equivalent to the three following 
conditions:
\begin{equation}\label{cond1}
A_{11} A_{22}-A_{12}^2=I,
\end{equation}
\begin{equation}\label{cond2}
A_{11} A_{12}^T=A_{12} A_{11},
\end{equation}
\begin{equation}\label{cond3}
A_{12}^T A_{22}=A_{22} A_{12}.
\end{equation}

From (\ref{cond1}) we get $A_{22}=A_{11}^{-1}+ (A_{11}^{-1} A_{12}) A_{12}$. We can rewrite (\ref{cond2}) as 
$A_{12}^T A_{11}^{-1}=A_{11}^{-1}  A_{12}$. Thus, we have
$A_{22}=A_{11}^{-1}+ A_{12}^TA_{11}^{-1}A_{12}$, which together with (\ref{S1}) leads to  (\ref{S2}). 
\end{proof}

\medskip

We propose methods for computing  symplectic $LL^T$ factorization of a given  symmetric positive definite symplectic matrix $A$, 
where $L$ is symplectic and  partitioned as in  (\ref{L}).  We apply  the Cholesky and the Reverse Cholesky decompositions. 
Practical algorithm for the Reverse Cholesky decomposition is described in Section $2$, see Remark \ref{remark2}.  

\begin{theorem}\label{thm1}
Let $M \in \mathbb R^{m \times m}$ be a  symmetric positive definite matrix. 
\begin{description}
\item[(i)] Then there exists a unique lower triangular matrix $L \in  \mathbb R^{m \times m}$ with positive diagonal entries  such that 
$M=L L^T$ (Cholesky decomposition).

\item[(ii)] Then there exists a unique upper triangular matrix 
$U \in  \mathbb R^{m \times m}$ with positive diagonal entries  such that $M=U U^T$ (Reverse Cholesky decomposition).
\end{description}
\end{theorem}

\begin{proof} We only need to prove {\bf (ii)}. Using the fact {\bf (i)} for the inverse of $M$, we get $M^{-1}=\hat L {\hat L}^T$, where $\hat L$ is 
a lower triangular matrix with positive diagonal entries. Then $M=(\hat L {\hat L}^T)^{-1}=U U^T$ where $U={\hat L}^{-T}$. Clearly, $U$ is upper triangular with positive entries, and 
$U$ is unique.
\end{proof}

\medskip

Based on Theorem \ref{thm1}, we prove the following result on symplectic $L L^T$ factorization (see  \cite{Dopico}, Theorem 5.2).

\begin{theorem}\label{thm2}
Let  $A \in \mathbb R^{2n \times 2n}$  be a symmetric positive definite symplectic matrix of the form
\begin{equation}\label{form A}
A=\left(
\begin{array}{cc}
 A_{11}  &  A_{12} \\
 A_{12}^T &   A_{22} 
\end{array}
\right).
\end{equation}

If $A_{11}=L_{11} L_{11}^T$ is the Cholesky decomposition of $A_{11}$, then $A=L L^T$, in which 
\begin{equation}\label{eqs L}
L=\left(
\begin{array}{cc}
 L_{11} &   0 \\
 L_{21} &  L_{22}
\end{array}
\right)=
\left(
\begin{array}{cc}
 L_{11} &   0 \\
 (L_{11}^{-1} A_{12})^T &  L_{11}^{-T}
\end{array}
\right)
\end{equation}
is symplectic.

If  $S$ is the Schur complement of $A_{11}$ in $A$, defined in (\ref{S1}, and  $S=U U^T$ is the Reverse Cholesky decomposition of $S$, 
then $L_{22}= L_{11}^{-T} = U$.
\end{theorem}

\begin{proof} We can write
\[
L L^T=
\left(
\begin{array}{cc}
L_{11}  L_{11}^T &  L_{11}  L_{21}^T  \\
 ( L_{11}  L_{21}^T)^T & L_{21} L_{21}^T +  L_{22} L_{22}^T
\end{array}
\right).
\]
This gives   the identities
\[
A_{11}= L_{11}  L_{11}^T, \quad A_{12}= L_{11}  L_{21}^T, \quad A_{22}=L_{21} L_{21}^T+L_{22} L_{22}^T.
\]

Clearly, $L_{21}^T=L_{11}^{-1} A_{12}$, and $S=A_{22}-L_{21} L_{21}^T$ is the Schur complement of $A_{11}$ in $A$.
Moreover, $S=L_{22} L_{22}^T$. If  $S=U U^T$ is the Reverse Cholesky decomposition of $S$ and $L_{22}$ is upper triangular, then $L_{22}=U$, by Theorem \ref{thm1}.
From  Lemma \ref{lemma5} we have  $S=A_{11}^{-1}$, hence  $S= L_{11}^{-T} L_{11}^{-1}$. 
Notice that $L_{11}^{-T}$ is upper triangular, so $U=L_{11}^{-T}$. 
 
It is easy to prove that $L$ in \ref{eqs L} is symplectic. It follows  from Lemma  \ref{lemma1} and (\ref{cond2}).
\end{proof}

\medskip

The paper is organized as follows.  Section $2$ describes  Algorithms $W_1$ and $W_2$. 
Section $3$ present both theoretical and practical computational issues.  Section $4$ is devoted to numerical experiments and comparisons of the methods.
Conclusions are given in Section $5$.

\medskip

\section {Algorithms}\label{algorithms}

We apply Theorem \ref{thm2} to develop two algorithms for finding the symplectic $LL^T$ factorization. They differ only in a way of computing the matrix $L_{22}$.
Algorithm $W_1$ is based on Theorem 5.2 from \cite{Dopico}. We propose Algorithm $W_2$, which can be used for symmetric positive definite matrix $A$, 
not necessarily symplectic. However, if $A$ is additionally symplectic then the factor $L$ is also symplectic as well.

\begin{itemize}
\item[] {\bf Algorithm $W_1$}

Given a symmetric positive definite symplectic matrix $A \in   \mathbb R^{2n \times 2n}$. This algorithm computes the symplectic  $LL^T$ factorization $A=LL^T$, where
$L$ is symplectic and has a form 
\[
L=\left(
\begin{array}{cc}
 L_{11} &   0 \\
 L_{21} &  L_{22}
\end{array}
\right).
\]
\begin{itemize}
\item Find the Cholesky decomposition  $A_{11}=L_{11}L_{11}^T$.
\item Solve the multiple  lower triangular system   $L_{11}  L_{21}^T=A_{12}$ by forward  substitution.
\item Solve the lower triangular system  $L_{11} X=I$ by forward substitution, i.e. computing  each column of $X=L_{11}^{-1}$  independly,  using forward substitution.     
\item Take $L_{22}=X^{T}$. 
\end{itemize}
\noindent
{\bf Cost:} $\frac{5}{3} n^3$ flops.

\item[] {\bf Algorithm $W_2$}

Given a symmetric positive definite symplectic matrix $A \in  \mathbb R^{2n \times 2n}$. This algorithm computes the symplectic $LL^T$ factorization $A=LL^T$, where
$L$ is symplectic and has a form 
\[
L=\left(
\begin{array}{cc}
 L_{11} &   0 \\
 L_{21} &  L_{22}
\end{array}
\right).
\]
\begin{itemize}
\item Find the Cholesky factorization $A_{11}=L_{11}L_{11}^T$.
\item Solve the multiple  lower triangular system   $L_{11}  L_{21}^T=A_{12}$ by forward  substitution.
\item Compute the Schur complement $S=A_{22}- L_{21} L_{21}^T$. 
\item Find the Reverse Cholesky decomposition   $S=L_{22}\, L_{22}^T$, where $L_{22}$ is upper triangular matrix with positive diagonal entries. 
\end{itemize}
\end{itemize}
\noindent
{\bf Cost:} $\frac{8}{3} n^3$ flops.

\medskip

\begin{remark} \label{remark2}
The Reverse Cholesky decomposition  $M=UU^T$ of a  symmetric  positive definite matrix $M \in   \mathbb R^{m \times m}$ can be treated as the  Cholesky decomposition of the matrix
$M_{new}= P^T M P$, where $P$ is the permutation matrix comprising the identity matrix with its column in reverse order. If $M_{new}=L L^T$, where $L$ is lower triangular 
(with positive diagonal entries), then $M=UU^T$, with $U=P L P^T$ being  upper triangular (with positive diagonal entries). 

For example, for $m=3$ we have
\[
P=\left(
\begin{array}{ccc}
 0  &  0 & 1  \\
 0  &  1 & 0  \\ 
 1  &  0 & 0  
\end{array}
\right), 
\quad P^T M P=
\left(
\begin{array}{ccc}
 m_{33} &  m_{32} & m_{31}  \\
 m_{23} &  m_{22} & m_{21}  \\
m_{13} &  m_{12} & m_{11}
\end{array}
\right),
\]
and
\[
L=\left(
\begin{array}{ccc}
  l_{11} &  0 & 0  \\
 l_{21} &  l_{22} & 0  \\
l_{31} &  l_{32} & l_{33}
\end{array}
\right),
\quad U=\left(
\begin{array}{ccc}
  l_{33} &  l_{32} & l_{31}  \\
 0 &  l_{22} & l_{21}  \\
0 &  0 & l_{11}
\end{array}
\right).
\]

\medskip

We use the following \textsl{MATLAB} code:
 
\begin{verbatim}
function U = reverse_chol(M)
% U = reverse_chol(M)
% The Reverse Cholesky decomposition M=U U',
% where U is upper triangular with positive diagonal entries.
% Here M(mxm) is a symmetric positive definite matrix.
% 
m=max(size(M)); U=zeros(m); 
p=m:-1:1; 
M_new=M(p,p);
L=chol(M_new,'lower'); % Cholesky decomposition
U=L(p,p);
end
\end{verbatim}
\end{remark}

\medskip

\section{Theoretical and practical computational issues}

In this work,  for any matrix $X \in \mathbb R^{m \times m}$,   $\normtwo{X}$ denotes  the 2-norm (the spectral norm) of $A$, 
and $\kappa_2(X)=\normtwo{X^{-1}} \cdot \normtwo{X}$ is the condition number of  a nonsingular matrix $X$.

This section mainly addresses  the problem of measuring the departure of a given matrix from symplecticity. We also touch a few aspects of 
numerical stability of Algorithms $W_1$ and $W_2$. However, this topic exceeds the scope of this paper. 

First we introduce  the loss of symplecticity (absolute error)  of $X \in \mathbb R^{2n \times 2n}$ as
\begin{equation}\label{DeltaX}
\Delta (X) = \normtwo{X^TJX-J}.
\end{equation}

Clearly,  $\Delta (X)=0$ if and only if $X$ is symplectic. 
If $X \in \mathbb R^{2n \times 2n}$  is symplectic then $X^{-1}=J^T X^T J$, 
and the condition number of $X$ equals  $\kappa_2(X)=\normtwo{X}^2$.
However,  in practice  $\Delta (X)$ hardly ever  equals  $0$. 

\medskip

\begin{lemma}\label{lemma6}
Let $X \in \mathbb R^{2n \times 2n}$ satisfy  $\Delta (X)<1$. 
Then $X$ is nonsingular and we have 
\begin{equation}\label{kappaX}
\kappa_2 (X) \leq \frac{\normtwo{X}^2}{1-\Delta (X)}.
\end{equation}
\end{lemma}

\begin{proof}
Assume that $\Delta (X)<1$. We first prove that $\det X \neq 0$.    

Define  $F=X^TJX-J$.  Since $J^T=-J$ and $J^2=-I_{2n}$, we have the identity
\begin{equation}\label{JF}
X^TJX=J (I_{2n}-JF).
\end{equation}

Since $J$ is orthogonal, we get   $\normtwo{JF}=\normtwo{F}=\Delta (X)<1$, hence  the matrix $I_{2n}-JF$ is nonsigular. 
Then (\ref{JF}) and the property $\det J=1$  leads to  $(\det X)^2=\det (X^TJX)=\det (I_{2n}-JF) \neq 0$. Therefore, $\det X \neq 0$.

To estimate  $\kappa_2 (X)$,  we rewrite  (\ref{JF}) as 
\begin{equation}\label{invX}
 X^{-1}= (I_{2n}-JF)^{-1} (J^TX^TJ). 
\end{equation}

Taking norms we obtain 
\[
\normtwo{X^{-1}}\leq \normtwo{(I_{2n}-JF)^{-1}} \, \normtwo{J^TX^TJ} \leq \frac{\normtwo{X}}{1-\normtwo{JF}}.
\]

This together with  $\normtwo{JF}=\Delta (X)$ establishes the formula  (\ref{kappaX}).
The proof is complete.
\end{proof}

\medskip

Now we show  that  the assumption  $\Delta (X)<1$  of  Lemma \ref{lemma6} is crucial.

\begin{lemma}\label{lemma7}
For every $t \ge 1$ and every natural number $n$ there exists a singular matrix $X \in \mathbb R^{2n \times 2n}$ such that  $\Delta (X)=t$. 
\end{lemma}

\medskip

\begin{proof}
The proof consists in the construction of such matrix $X$.

Define 
\[
X=\left(
\begin{array}{cc}
 D &   0 \\
 0 &  -D
\end{array}
\right),
\]
where $D=\sqrt {t-1} \, diag(1,0, \ldots,0)$.  Clearly, $\det X=\det D \, \det(-D)=0$.

Then  we have
\[
X^TJX-J=\left(
\begin{array}{cc}
 0 &   -(D^2+I_n) \\
 D^2+I_n &  0
\end{array}
\right).
\]

Therefore, $\Delta (X)=\normtwo{D^2+I_n}= \normtwo{diag(t,1, \ldots, 1)}=t$.
This completes the proof.
\end{proof}

\medskip

\begin{lemma}\label{lemma8} 
Let $A \in \mathbb R^{2n \times 2n}$ be a symplectic matrix.  Suppose that the perturbed matrix  $\hat A=A+E$ satisfies
\begin{equation}\label{normE}
\normtwo{E}\leq \epsilon \normtwo{A}, \quad 0<\epsilon<1.
\end{equation}

Then $\hat A \neq 0$ and 
\begin{equation}\label{boundE}
\Delta (\hat A) \leq \normtwo{\hat A}^2 \, \, (2 \epsilon+{\cal O}(\epsilon^2)).
\end{equation}
\end{lemma}

\medskip

\begin{proof}
We begin by proving that $\normtwo{\hat A}>0$ for  $0< \epsilon <1$. Note that
$\normtwo{A+E} \ge \normtwo{A}-\normtwo{E}$. This together with (\ref{normE}) leads to
\begin{equation}\label{normhat}
\normtwo{\hat A} \ge (1-\epsilon) \normtwo {A}>0,
\end{equation}
hence  $\hat A \neq 0$.

It remains to estimate $\Delta (\hat A)$. For simplicity of notation, we define
\[ 
F=(A+E)^T J (A+E)-J.
\]  
Since $A$ is symplectic, we get $A^TJA-J=0$, hence $F=A^TJ E+E^T JA+E^T JE$. Taking norms we obtain
\[
\Delta (\hat A)=\normtwo{F} \leq 2 \normtwo{A} \normtwo{E} + {\normtwo{E}}^2.
\]

Applying (\ref{normE}) yields 
\begin{equation}\label{bound1}
\Delta (\hat A) \leq  \normtwo{A}^2 \, \, (2 \epsilon+\epsilon^2).
\end{equation}

From  (\ref{normE}) we deduce that $\normtwo{\hat A}=\normtwo{A} (1+\beta)$, where $|\beta| \leq \epsilon$. 
This together with  (\ref{normE}) and (\ref{bound1}) gives
\[
\Delta (\hat A) \leq \normtwo{\hat A}^2 \, \frac{(2 \epsilon+\epsilon^2)}{(1-\epsilon)^2},
\]
which  completes the proof.
\end{proof}

\medskip

According to (\ref{boundE})  we  introduce the loss of symplecticity (relative error) of  nonzero matrix $A \in \mathbb R^{2n \times 2n}$  as 
\begin{equation}\label{sympA}
sympA = \frac{\normtwo{A^TJA-J}}{\normtwo{A}^2}.
\end{equation}

\medskip

\begin{remark}\label{remark3}
Assume that $A$ is symplectic. Then  we have $A^TJA=J$, so taking norms  we obtain  
\[  
1=\normtwo{J} \leq \normtwo{A^T} \normtwo{J} \normtwo{A} ={\normtwo{A}}^2.
\] 

We see that $\normtwo{A} \ge 1$ for every symplectic matrix $A$.
Therefore, under  the hypotheses of Lemma \ref{lemma8} and applying (\ref{normhat}) we get the inequality
\begin{equation}\label{deltasymp}
\Delta (\hat A) \ge  (1-\epsilon)^2  \,\, {\normtwo{A}}^2  \,\, symp {\hat A}. 
\end{equation} 

If   $\normtwo{A}$ is large and $\hat A$ is close to $A$, then $symp {\hat A} << \Delta (\hat A)$.   This property is highligted in our numerical experiments in Section $4$.
\end{remark}
\medskip

\begin{proposition}\label{thm symp}
Let  $\tilde L  \in \mathbb R^{2n \times 2n}$  be  the computed factor of the symplectic factorization   $A=LL^T$, where 
$A \in \mathbb R^{2n \times 2n}$  is a symmetric positive definite symplectic matrix. 

Define 
\begin{equation}\label{def F}
F={\tilde L}^T J \tilde L - J. 
\end{equation}

Partition $\tilde L$ and $F$  conformally  with $J$  as
\begin{equation}\label{L and F}
\tilde L=\left(
\begin{array}{cc}
 \tilde L_{11} &   0 \\
 \tilde L_{21} & \tilde   L_{22}
\end{array}
\right),
\quad 
F=\left(
\begin{array}{cc}
 F_{11} &  F_{12} \\
 F_{21} & F_{22}
\end{array}
\right).
\end{equation}

Then $F_{21}=-{F_{12}}^T$,  $F_{22}=0$ and 
\begin{equation}\label{Fij}
F_{11}= {\tilde L_{11}}^T  \tilde L_{21}-{\tilde L_{21}}^T \tilde L_{11}, \quad  F_{12}={\tilde L_{11}}^T  \tilde L_{22}-I_n.
\end{equation}

Moreover,  the loss of symplecticity $\Delta (\tilde L)$ can be bounded as follows
\begin{equation}\label{boundF}
\max \{\normtwo{F_{11}}, \normtwo{F_{12}} \} \leq \Delta (\tilde L) \leq 2 \, \max \{\normtwo{F_{11}}, \normtwo{F_{12}}\}.
\end{equation}
\end{proposition}

\medskip

\begin{proof}
It is easy to check that  $F$ is a skew-symmetric matrix satisfying (\ref{Fij}), with $F_{22}=0$. 
Notice that $\Delta (\tilde L)= \normtwo{F}$. 
It remains to prove (\ref{boundF}). 

Write $F$ in a form $F=F_1+F_2$, where
\[
F_1=\left(
\begin{array}{cc}
 F_{11} &  0 \\
 0 & 0
\end{array}
\right),
\quad 
F_2=\left(
\begin{array}{cc}
 0 &  F_{12} \\
-{F_{12}}^T & 0
\end{array}
\right).
\]

It is obvious that  $\normtwo{F_1}=\normtwo{F_{11}}$ and  $\normtwo{F_2}=\normtwo{F_{12}}$. 
By property of 2-norm, it follows that  $\normtwo{F_{ij}} \leq \normtwo{F}$ for all $i, j=1,2$. 

On the other hand,  we get
\[
\normtwo{F} \leq \normtwo{F_1}+\normtwo{F_2} \leq 2 \,  \max \{\normtwo{F_{1}}, \normtwo{F_{2}}\}.
\]

This completes the proof.
\end{proof}

\medskip

\begin{remark}\label{remark4}
If  Algorithm $W_1$  runs to completion in floating-point arithmetic,  then  $\tilde L_{22}={\tilde L_{11}}^{-T} + {\cal O}(\eps)$, where $\eps$ is  machine precision.
See  \cite{higham:2002}, pp. 263--264, where  the detailed error analysis of methods for inverting triangular matrix was given. 
Notice that   $\normtwo{F_{12}}$ defined by (\ref{Fij}) depends only on conditioning of $A_{11}$, the submatrix of $A$.  Since $A$ is symmetric positive definite it follows that 
$\kappa_2(A_{11}) \leq \kappa_2(A)$.  However, the loss of simplecticity of   $\tilde L$ from Algorithm $W_2$ can be much larger than for Algorithm $W_1$, see our examples presented in Section \ref{tests}.

Notice  that  $F_{11}$  defined by (\ref{Fij}) remains  the same for both Algorithms $W_1$ and $W_2$. 
 \end{remark}

\medskip

Now we explain what we mean by {\em numerical stability} of algorithms for  computing $LL^T$ factorization.  

The precise definition is the following.

\begin{definition}\label{stab}
An algorithm $W$ for computing the $LL^T$ factorization of a given symmetric positive definite matrix $A \in \mathbb R^{2n \times 2n}$ is {\em numerically stable}, if the computed matrix
$\tilde L \in \mathbb R^{2n \times 2n}$, partitioned as in (\ref{L and F}),  is the exact factor of the $LL^T$ factorization of a  slightly perturbed matrix $A+\delta A$, with $\normtwo{\delta A}\leq \eps c \normtwo{A}$, where $c$ is a small constant depending upon $n$, and $\eps$ is machine precision. 
\end{definition}

In practice,   we can compute  the decomposition error
\begin{equation}\label{dec}
{dec}= \frac{\normtwo{A-\tilde L {\tilde L}^T}}{\normtwo{A}}.
\end{equation}

If $dec$ is of order $\eps$ then this is the best result we can achieve in floating-point arithmetic. 
We emphasize that here we apply  numerically stable Cholesky decomposition  of  symmetric positive definite matrix $A_{11}$ (see  Theorem 10.3 in \cite{higham:2002}, p. 197), and also numerically stable Reverse Cholesky decomposition of the Schur complement $S$ (defined by (\ref{S2})) applied in   Algorithm $W_2$. Notice that Lemma \ref{lemma5} implies that $\kappa_2(S)=\kappa_2(A_{11})$. For general symmetric positive definite matrix $A$ we have  a weaker bound: $\kappa_2(S) \leq \kappa_2(A)$, see \cite{Demmel}. 

\medskip

\section{Numerical experiments}\label{tests}

In this section we present numerical tests that show the comparison of Algorithms $W_1$ and $W_2$.  
All tests were performed in  \textsl{MATLAB} ver. R2021a,  with machine precision $\eps \approx 2.2 \cdot 10^{-16}$.

We report the following statistics:
\begin{itemize}
\item $\Delta(A) = \normtwo{A^TJA-J}$ (loss of symplecticity (absolute error) of $A$),
\item ${sympA}= \frac{\normtwo{A^TJA-J}}{\normtwo{A}^2}$ (loss of symplecticity (relative error) of $A$),
\item ${dec}_{Algorithm}= \frac{\normtwo{A-\tilde L {\tilde L}^T}}{\normtwo{A}}$ (decomposition error),
\item ${\Delta L}_{Algorithm}= \normtwo{\tilde L^TJ \tilde L-J}$ (loss of symplecticity (absolute error) of $\tilde L$),
\item ${sympL}_{Algorithm}= \frac{\normtwo{{\tilde L}^TJ\tilde L-J}}{\normtwo{\tilde L}^2}$ (loss of symplecticity (relative error) of $\tilde L$),
\item $\normtwo{F_{11}}$ and  $\normtwo{F_{12}}$ defined by (\ref{def F})--(\ref{Fij}).
\end{itemize}

\medskip

\begin{example}\label{example1}

In the first experiment we take  $A= S^T\, S$, where $S$ is a symplectic matrix, which was also used in   \cite{Benzi} and  \cite{Tam}:
\begin{equation}\label{S4x4}
S=S(t)=\left(
\begin{array}{cccc}
\cosh{t}   &   \sinh{t} & 0 &   \sinh{t}\\
 \sinh{t} & \cosh{t} &  \sinh{t} & 0 \\
0 & 0 & \cosh{t}   &  - \sinh{t} \\
0 & 0 & -\sinh{t}   &   \cosh{t} \\
 \end{array}
\right),
\quad t \in  \mathbb R.
\end{equation}

The results are contained in Table $1$. 
We see that Algorithm $W_1$ produces unstable result  $\tilde L$, opposite to Algorithm $W_2$.
\medskip

\noindent

\begin{table}[!h]\label{tabelka1}
\caption{The results for Example \ref{example1} and $A=S^T \,S$, where $S$ is defined by (\ref{S4x4}).}

\medskip
\begin{center}
\begin{tabular}{lcccc}
\hline
 t & $\pi$ & $\frac{3}{2} \pi$ & $2 \pi$ &  $\frac{5}{2} \pi$\\
\hline 
$\kappa_2 (A)$ & 4.4738e+05 & 2.3991e+08 & 1.2848e+11 & 6.8988e+13 \\
$\kappa_2 (A_{11})$ & 2.8675e+05 & 1.5355e+08 & 8.2227e+10 & 4.4063e+13 \\
${dec}_{W_1}$ &  1.2107e-11 & 4.5114e-09 & 1.0703e-06 & 0.0012\\ 
${dec}_{W_2}$ & 8.4985e-17 & 1.0127e-16 & 8.1196e-17 & 5.6141e-17\\ 
${sympA}$ &  6.3656e-17 & 5.6888e-17 & 6.9038e-17  & 4.5934e-17\\
${sympL}_{W_1}$   & 2.7064e-16 & 7.6363e-14 & 1.2318e-12 & 1.3975e-10\\
${sympL}_{W_2}$ &  5.1473e-14 & 2.2866e-13 & 8.1877e-12 & 1.8220e-10\\
$\Delta (A)$ &  2.8478e-11  &  1.3648e-08  &  9.5688e-06  & 0.0032\\		
$\Delta L_{W_1}$ &  1.8102e-13 & 1.1828e-09 &    4.4153e-07 &   0.0012\\
$\Delta L_{W_2}$ &  4.2038e-11  & 3.5417e-09  & 1.0328e-06 &   0.0015\\
$\normtwo{F_{11}}$ & 1.7186e-13 & 1.1828e-09 &  4.4153e-07 &   0.0012\\
$\normtwo{F_{12}}$ \text{from  $W_1$}   &  5.6843e-14 &  0 & 0 &   0\\
$\normtwo{F_{12}}$  \text{from $W_2$}  & 4.2038e-11 & 3.1147e-09 & 8.7430e-07  & 9.5561e-04\\
 \hline
\end{tabular}
\end{center}
\end{table}
\end{example}

\medskip

\begin{example}\label{example2}

For comparison, in the second  experiment we use the same matrix $S$ and repeat the calculations for the inverse of $A$ from the Example \ref{example1}.
Since  $\kappa_2(A^{-1})=\kappa_2(A)$, we see that the condition numbers of  $A$ is  the same in both Examples $1$ and $2$. 
However, here $A_{11}$ is perfectly well-conditioned, opposite to the previous Example \ref{example1}.
The results are contained in Table $2$. 
Now Algorithm $W_1$ produces numericall stable result $\tilde L$, like Algorithm $W_2$. 
We observe that for large values of  $\Delta A$ (in the last columns of Tables $1$ and $2$) the loss of simplicticity of computed $\tilde L$ is significant. 
\medskip

\noindent 

\begin{table}\label{tabelka2}
\caption{The results for Example \ref{example2} and  $A=(S^T \,S)^{-1}$, where $S$ is defined by (\ref{S4x4}).}

\medskip
\begin{center}
\begin{tabular}{lcccc}
\hline
 t & $\pi$ & $\frac{3}{2} \pi$ & $2 \pi$ &  $\frac{5}{2} \pi$\\
\hline 
$\mathrm{\kappa_2 (A)}$  & 4.4738e+05 &   2.3991e+08 & 1.2848e+11 & 6.9042e+13\\
$\mathrm{\kappa_2 (A_{11})}$ & 5.0149 & 5.0006 & 5.0001 & 4.9996 \\
${dec}_{W_1}$ &  1.3751e-16 & 2.7434e-16 & 4.5207e-16 & 1.1685e-16\\ 
${dec}_{W_2}$        & 8.4985e-17 & 9.1095e-17 & 4.0598e-17 & 1.1892e-16\\ 
${sympA}$ &  5.7906e-17 & 6.3647e-17 & 1.1614e-16 & 6.4968e-17\\
${sympL}_{W_1}$  &  0 & 1.1744e-16 & 8.1196e-17 & 8.4218e-17\\
${sympL}_{W_2}$  & 4.9186e-14 & 4.6509e-13 & 2.0383e-11 & 1.5995e-10\\
$\Delta (A)$ &    2.8478e-11  & 1.3648e-08  &  9.5688e-06  &   0.0032\\
$\Delta L_{W_1}$ &   0 &    1.8190e-12 &   2.9104e-11 &    6.9849e-10\\
$\Delta L_{W_2}$ &  3.2899e-11  &  9.7380e-09  & 3.1494e-06 &   0.0011\\
$\normtwo{F_{11}}$ &  0 & 1.8190e-12 & 2.9104e-11 &    6.9849e-10\\
$\normtwo{F_{12}}$ \text{from $W_1$}   &  0 &   1.5701e-16 & 1.1102e-16 &   1.1102e-16\\
$\normtwo{F_{12}}$  \text{from $W_2$}  & 3.2899e-11 &  9.7380e-09 & 3.1494e-06 &    0.0011\\
\hline
\end{tabular}
\end{center}
\end{table}
\end{example}

\medskip

\begin{example}\label{example3}

Here $A(10 \times 10)$ is generated as follows
\begin{verbatim}
randn('state',0);
n=5; s=3;
A=gener_symp2(n,s)+t*hilb(2*n);
\end{verbatim}

Random matrices of entries are from the distribution $N(0,1)$. They were generated by \textsl{MATLAB} function "randn".
Before each usage the random number generator was reset to its initial state.

Here we use Lemmas \ref{lemma2}--\ref{lemma3} to create the following \textsl{MATLAB} functions: 
\begin{itemize}
\item function  for generating orthogonal symplectic matrix $Q(2n \times 2n)$:
\begin{verbatim}
function [Q] = orth_symp(n)
% [Q] = orth_symp(n)
%
[U,~]=qr(complex(randn(n),randn(n)));
C=real(U);
S=imag(U);
Q=[C, S;-S,C];
end
\end{verbatim}

\item and function   for generating symmetric positive definite symplectic matrix  $S(2n \times 2n)$
with prescribed condition number $\kappa_2(S)=10^{2s}$

\begin{verbatim}
function [S]=gener_symp2(n,s)
% function [S]=gener_symp2(n,s)
% S=U G U', where U is orthogonal symplectic matrix.
% G=diag(D,inv(D)), D=diag(d), d=(d_1,...,d_n).
% Here  cond(S)=cond(G)=10^(2s).
% 
d=flip(logspace(0,s,n));
g=[d,1./d]; G=diag(g);
U=orth_symp(n); 
S=U*G*U'; S=(S+S')/2;
end
\end{verbatim}
\end{itemize}

The results are contained in Table $3$. 
However, the results of $\normtwo{F_{12}}$  from Algorithm $W_2$ are catastrophic in comparison with the values received from Algorithm $W_2$. 
Here $A_{11}$ is quite well-conditioned, but  the departure of $A$ from simplicticity conditions is very large.   

\medskip

\noindent 

\begin{table}\label{tabelka3}
\caption{The results for Example \ref{example3}.}

\medskip
\begin{center}
\begin{tabular}{lcccc}
\hline
 t & $0$ & $10^{-6}$  & $\frac{1}{2}$ &  $1$\\
\hline 
$\kappa_2 (A)$ & 1.0000e+06 & 9.9990e+05 & 1.0215e+05 & 8.6253e+04 \\
 $\kappa_2 (A_{11})$ & 620.6887 & 620.6885 & 534.9257 & 470.6434 \\
 ${dec}_{W_1}$ &  9.0334e-17 & 3.9861e-09 & 0.0019 & 0.0037\\ 
${dec}_{W_2}$ & 6.8325e-17 & 6.4834e-17 & 7.7533e-17 & 5.9520e-17\\ 
 ${sympA}$ &  3.7658e-17 & 2.2200e-10 & 1.1103e-04  & 2.2211e-04\\  
 ${sympL}_{W_1}$   & 6.3637e-16 & 4.9454e-09 &0.0023  & 0.0043\\
 ${sympL}_{W_2}$ & 7.5755e-15 & 7.9023e-08 & 0.0081 & 0.0118\\
 $\Delta (A)$ &  3.7658e-11  & 2.2200e-04  & 111.0506  &  222.2034\\		
$\Delta L_{W_1}$ &   6.3637e-13 &  4.9454e-06 &  2.2976 &   4.3139\\
$\Delta L_{W_2}$ &  4.8860e-12  & 7.9023e-05  &  8.1264 &  11.7686\\
$\normtwo{F_{11}}$ &  6.2070e-13 &    4.9454e-06 & 2.2976 &  4.3139\\
 $\normtwo{F_{12}}$ \text{from $W_1$}   &  1.7850e-15 &     7.8830e-16 &  4.6245e-16 &  5.4574e-16\\
  $\normtwo{F_{12}}$  \text{from  $W_2$}  &    4.8426e-12 &     7.8922e-05 & 7.8910 &  11.1795\\
      \hline
\end{tabular}
\end{center}
\end{table}
\end{example}

\medskip

\begin{example}\label{example4}

Now we apply Lemma \ref{lemma4} for creating our test matrices. We take
$A=PDP^T$, where 
\[
P=\left(
\begin{array}{cc}
 I_n &   0 \\
 C &  I_n
\end{array}
\right),
\quad 
D=\left(
\begin{array}{cc}
 G &   0 \\
 0 &  G^{-1}
\end{array}
\right),
\]
 
where $C$ is the Hilbert matrix and $\mathcal{B}$ is {\bf beta matrix}. 

Here $\mathcal{B}=\left(\frac{1}{\beta(i,j)}\right)$, where $\beta(\cdot,\cdot)$ is the $\beta$ function. 

By definition, 
	\[
	\beta(i,j)=\frac{\Gamma(i)\Gamma(j)}{\Gamma(i+j)},
	\]
where $\Gamma(\cdot)$ is the {\it Gamma function}.

$\mathcal{B}$ is symmetric totally positive matrix of integer.  More detailed information related to {\bf beta  matrix}  can be found in  \cite{Grover} and  \cite{higham:2021}.

Note  that generating  $A$ requires computing the inverse of the ill-conditioned Hilbert matrix. It  influences significantly on the quality of computed results in floating-point arithmetic.

The results are contained in Table $4$. 

\medskip

\noindent 

\begin{table}\label{tabelka4}
\caption{The results for Example \ref{example4}.}

\medskip
\begin{center}
	\begin{tabular}{lcccc}
		\hline
		n & 10 &16 & 20 & 24\\
		\hline
		$\kappa_2 (A)$ &1.1262e+06 &   6.2776e+09 &  1.9056e+12 &  5.6578e+14\\
		$\kappa_2 (A_{11})$ &  5.6043e+04 &  1.4639e+08 &  3.0158e+10 &  6.4618e+12\\
		${dec}_{W_1}$ &    6.3416e-17 &  1.0803e-12 &     2.8482e-11 &  1.5496e-10\\
		${dec}_{W_2}$  &    6.3329e-17  & 6.7428e-17 &  6.9001e-17 &   1.0661e-16\\
		${sympA}$ & 1.0928e-17  & 2.1256e-17 &  2.0100e-17 &  2.2716e-17\\
		${sympL}_{W_1}$ & 5.3818e-16  & 6.2006e-15&   4.3833e-14  & 3.3376e-13 \\
		${sympL}_{W_2}$ & 2.1743e-15 &  1.7424e-13 & 2.1389e-13 &  8.3257e-12 \\
		$\Delta (A)$ &  1.2307e-11  & 1.3344e-07  & 3.8304e-05  & 1.2844e-02\\		
		$\Delta L_{W_1}$ & 5.7112e-13 &  4.9129e-10 &  6.0509e-08 &  7.9364e-06\\
		$\Delta L_{W_2}$ & 2.3074e-12  & 1.3805e-08  & 2.9526e-07 &  1.9798e-04\\
		$\normtwo{F_{11}}$ &  5.6852e-13 &    4.9102e-10 &  6.0506e-08 & 7.9364e-06\\
           $\normtwo{F_{12}}$ \text{from $W_1$}   &  6.1156e-15 &  6.9449e-13 & 7.4234e-12 &   9.0109e-11\\
           $\normtwo{F_{12}}$  \text{from $W_2$}  &  2.2400e-12 &  1.3799e-08 & 2.8950e-07 & 1.9783e-04\\
        \hline
	\end{tabular}
\end{center}
\end{table}
\end{example}

\medskip

\begin{example}\label{example5}
The matrices $A(2n \times 2n)$ are  generated  for  $n=2:2:250$  by the following \textsl{MATLAB} code:
\begin{verbatim}
rand('state',0);  
randn('state',0);
d=rand(1,n);
U=orth_symp(n); 
g=[d,1./d]; G=diag(g);
A=U*G*U'; A=(A+A')/2;
\end{verbatim}

We applied  Lemma \ref{lemma3} for creating matrices of the form $A=UGU^T$, where $G$ is  a diagonal matrix,  and $U$ is an  orthogonal symplectic matrix, generated by the same  \textsl{MATLAB} function as in Example \ref{example3}.   

Figures  \ref{obrazek1}-\ref{obrazek3} illustrate the values of the statistics. We can see the differences between decomposition errors $dec$ (in favor of Algorithm $W_2$) and between the values $\Delta L$, in favor of Algorithm $W_2$. 

\medskip

\begin{figure}\label{obrazek1}
\includegraphics[width=6cm]{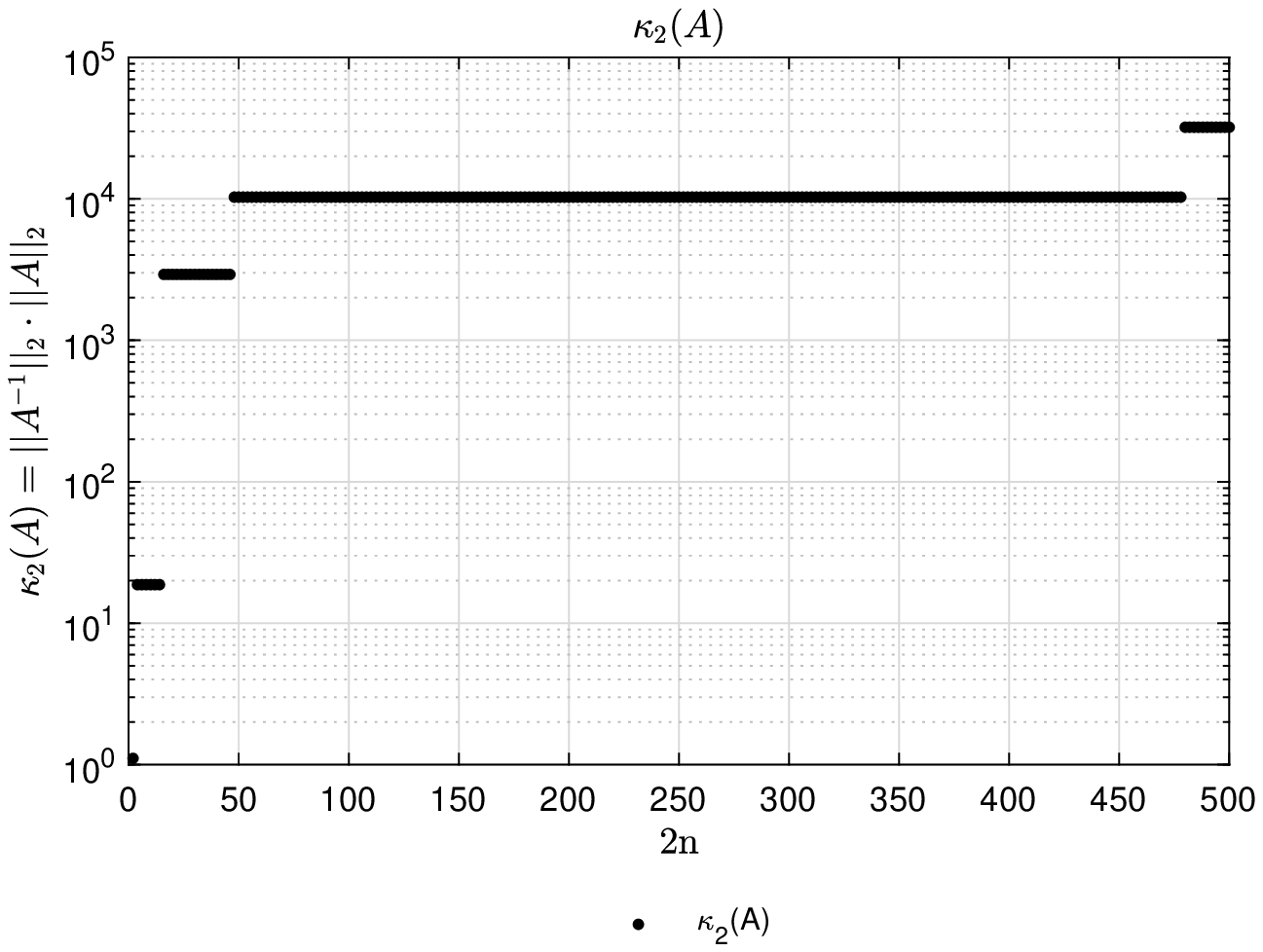}
\includegraphics[width=6cm]{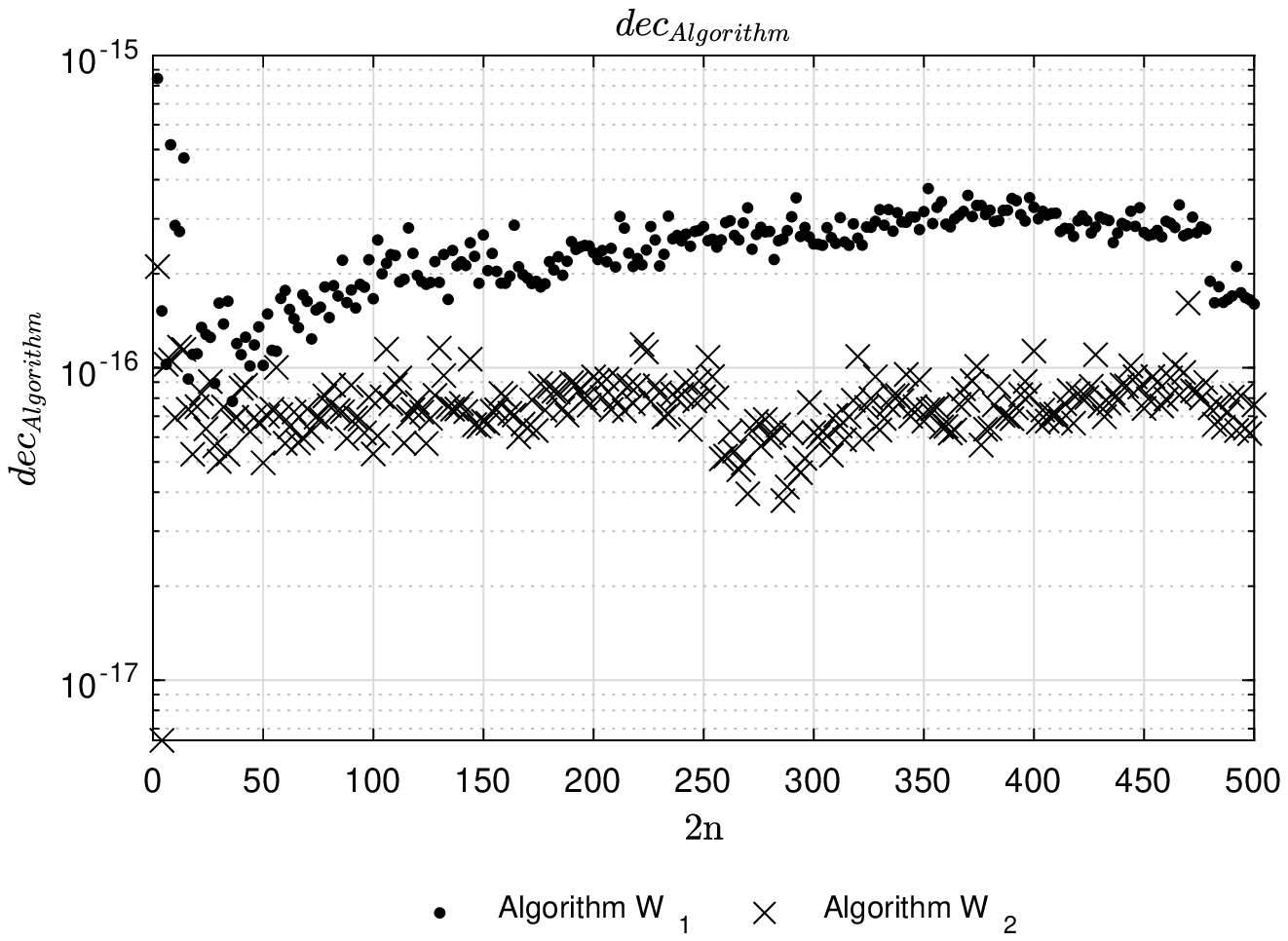}
\vskip-1.5ex
\caption{Condition numbers $\kappa_2(A)$ and decomposition errors for  Example \ref{example5}.}
\label{obrazek1}
\end{figure}

\begin{figure}\label{obrazek2}
\includegraphics[width=6cm]{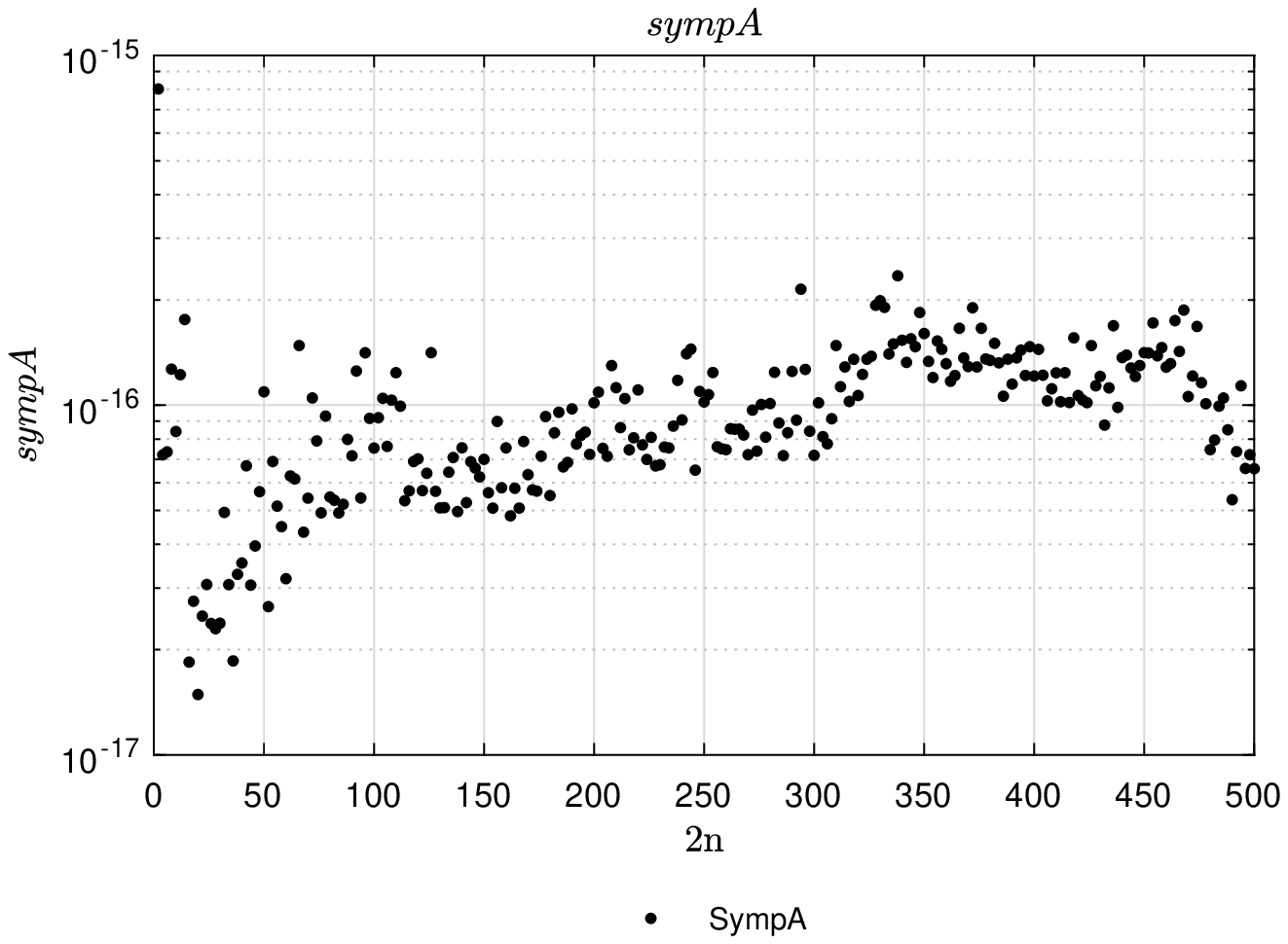}
\includegraphics[width=6cm]{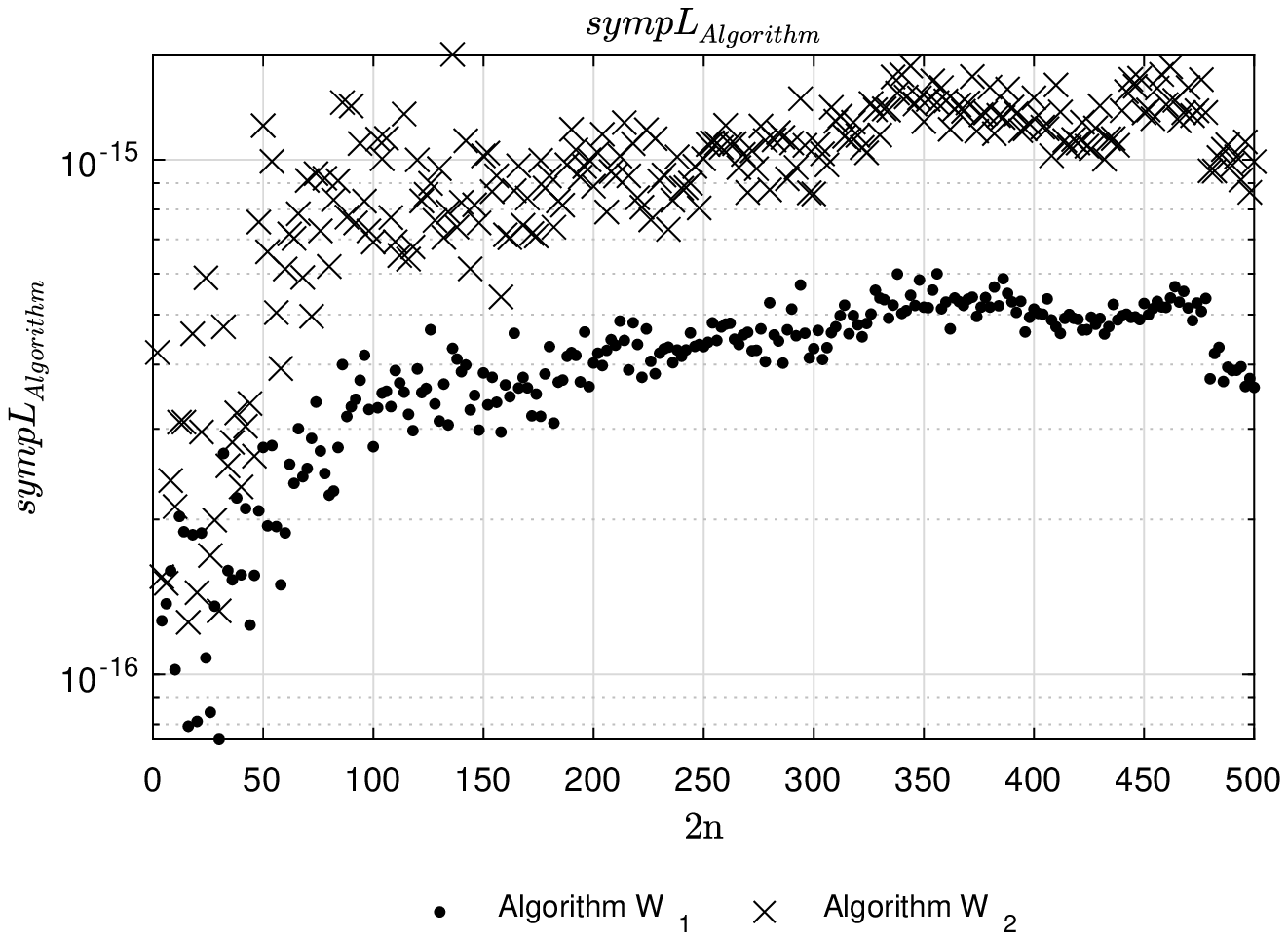}
\caption{The loss of symplecticity (relative errors)  for  Example \ref{example5}.}
\label{obrazek2}
\end{figure}

\vskip-1.5ex

\begin{figure}\label{obrazek3}
\includegraphics[width=6cm]{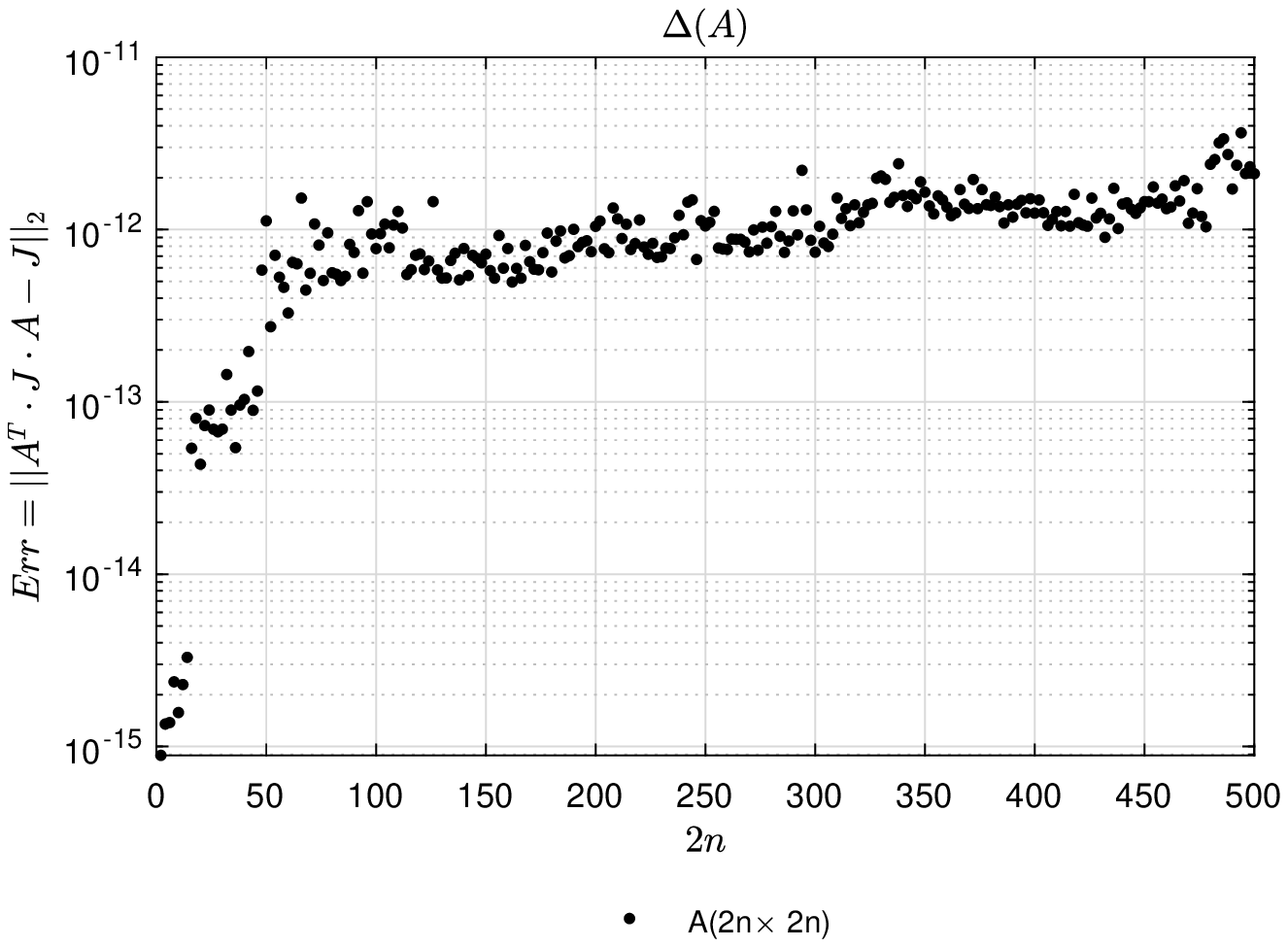}
\includegraphics[width=6cm]{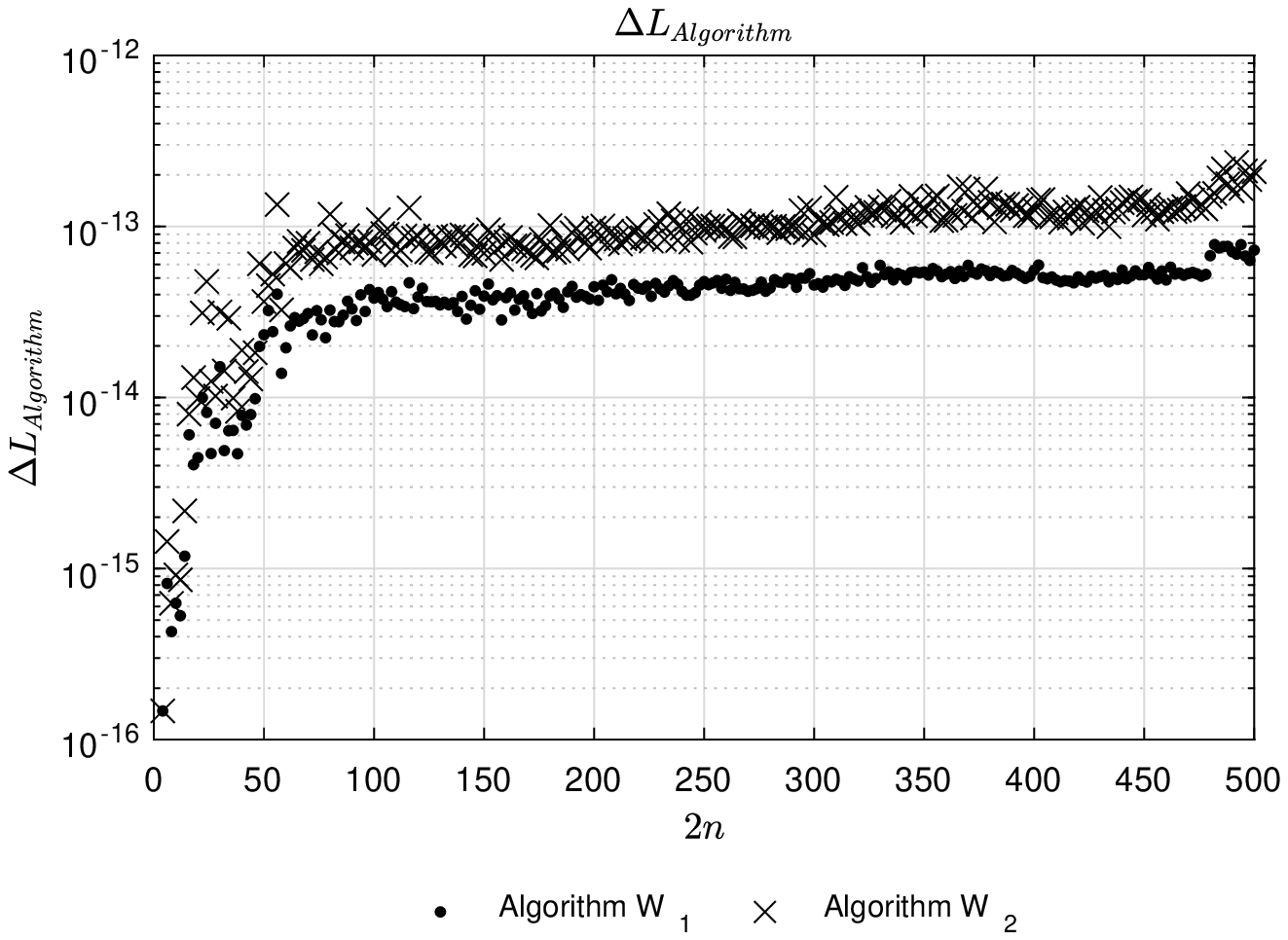}
\caption{The loss of symplecticity (absolute errors)  for  Example \ref{example5}.}
\label{obrazek3}
\end{figure}
\end{example}

\medskip

\section{Conclusions}

\begin{itemize}
\item We analyzed  two algorithms $W_1$ and $W_2$  for computing the symplectic $LL^T$ factorization of a given symmetric positive definite matrix $A(2n \times 2n)$. 
To assess their practical behaviour we performed numerical experiments. 

\item Algorithm $W_1$ is cheaper than Algorithm $W_2$. However, Algorithm $W_1$  is  unstable, in general, although  it works  very well for many test matrices.  The decomposition error (\ref{dec}) of the computed matrix $\tilde  L$  via Algorithm $W_1$  can be very  large.  In opposite,  in all our tests  Algorithm $W_2$ produces a numerically stable resulting matrices  $\tilde  L$ in floating-point arithmetic (in sense of Definition \ref{stab}).  Numerical stability of Algorithms $W_1$ and $W_2$   remains a topic of future work. 

\item Numerical tests presented  in Section  \ref{tests} give indication that  the loss of symplecticity  of  the computed matrix $\tilde  L$  from Algorithm $W_2$ can be much  larger  than obtained  from 
Algorithm $W_1$. We observe  that   the loss of symplecticity   of $\tilde L$ for both Algorithms $W_1$ and $W_2$ strongly depends on the distance from the simplicticity properties  (see Lemma \ref{lemma5}), and also on conditioning of  $A$ and its submatrix $A_{11}$. 
\end{itemize}

\end{document}